\documentclass[12pt]{amsart}
\usepackage[margin = 1.36in]{geometry}                % See geometry.pdf to learn the layout options. There are lots.
\usepackage{graphicx}
\usepackage{amssymb}
\usepackage{amsmath}
\usepackage{xcolor}

\usepackage{enumitem}

\usepackage{filecontents}

\newtheorem{thm}{Theorem}
\newtheorem{lma}[thm]{Lemma}
\newtheorem{cor}[thm]{Corollary}
\newtheorem{df}[thm]{Definition}
\newtheorem{prop}[thm]{Proposition}
\newtheorem{question}[thm]{Question}
\newtheorem{conjecture}[thm]{Conjecture}

\theoremstyle{remark}
\newtheorem{rmk}[thm]{Remark}

\newcommand{\om}{\omega}
\newcommand{\la}{\lambda}

\newcommand{\al}{\alpha}

\newcommand{\eps}{\varepsilon}

\newcommand{\R}{\mathbb{R}}
\newcommand{\Z}{\mathbb{Z}}
\newcommand{\Q}{\mathbb{Q}}
\newcommand{\C}{\mathbb{C}}
\newcommand{\F}{\mathbb{F}}

\newcommand{\D}{\mathbb{D}}

\newcommand{\mrm}[1]{\mathrm{#1}}
\newcommand{\cl}[1]{\mathcal{#1}}

\newcommand{\del}{\partial}

\renewcommand{\geq}{\geqslant}
\renewcommand{\leq}{\leqslant}

\title{On orderability and the chord conjecture}

\author{Egor Shelukhin}
\address{D\'epartement de Math\'ematiques et de Statistique, Universit\'e de
Montr\'eal, C.P. 6128 Succ. Centre-Ville, Montreal (Québec), H3C 3J7, Canada} 
\email{egor.shelukhin@umontreal.ca}

\date{Aug 10, 2026}                                           

\begin{document}

\begin{abstract}
We prove Arnol'd's chord conjecture for a large new class of contact manifolds: for every contact form and every closed Legendrian submanifold there exists a non-constant Reeb chord with endpoints on the Legendrian.
This class is characterized by contact non-orderability and rigidity of symplectizations. This proves the chord conjecture for Brieskorn manifolds, many prequantization spaces, and for all prequantization spaces under a mild topological condition on the Legendrians. Moreover, it provides a uniform upper bound on the length of the minimal chord. Our approach involves a new link between Mohnke's construction and contact Hofer geometry.
%Our arguments rely on the relation between non-orderability, the contact Hofer metric, and Chekanov-style estimates on the displacement energy.
\end{abstract}

\maketitle

\section{Introduction and main results}

Arnol'd has conjectured in 1986 \cite{ArnoldConj} that given an arbitrary closed Legendrian submanifold $K \subset S^3$ of the standard contact $3$-sphere, and any contact form $\al,$ there exists a non-constant $\al$ Reeb chord with boundary on $K.$ This conjecture was proven by Mohnke \cite{Mohnke}, whose proof generalizes to closed Legendrian submanifolds of $S^{2n-1}$ and other ideal contact boundaries of subcritical Weinstein domains (see also \cite{Zhou-minimal} for a generalization to the case of contact boundaries of Liouville domains with vanishing symplectic cohomology). Another major advance on this question is the work of Hutchings and Taubes \cite{HT1, HT2} establishing the conjecture for all closed Legendrian submanifolds of closed contact $3$-manifolds. We refer the reader to \cite{BCS-chord} for a further review of the literature, as well as to \cite{BC-chord, GZ-chord, Cant-strongchord} for very recent progress on the question.

%One of the main applications in this paper is the following result. 

To make our discussion precise, let us say that Arnol'd's chord conjecture holds for a closed Legendrian submanifold $K \subset Y$ of a closed contact manifold $(Y,\xi)$ if for every contact form $\al$ on $Y$ there exists a Reeb chord $c:[0,T] \to Y,$ $T>0,$ that is, a solution of the differential equation \[\dot{c}(t) = R(c(t))\] for all $t \in [0,T]$ with boundary conditions $c(0),c(T) \in K.$ Here $R$ is the Reeb vector field of $\al$ uniquely determined by the equations \[\al(R)=1,\; \iota_R d\al = 0.\]
We say that the conjecture holds for $Y$ if it holds for all such $K.$ %That is, for every closed Legedrian submanifold $K$ in $Y$ and every contact form $\al$ on $Y$ there exists a Reeb chord $c:[0,T] \to Y,$ $T>0,$ with $c(0),c(T) \in K.$ 

\medskip

The main application of the methods in this paper is the following result. Recall that a prequantization $(Y,\al)$ of a closed symplectic manifold $(M,\om)$ only exists if $[\om] \in H^2(M;\R)$ is an integral cohomology class, in the sense of belonging to the image of the natural map $H^2(M;\Z) \to H^2(M;\R),$ and it is a principal $S^1$-bundle $\pi: Y \to M$ with connection form $\al$ such that $d\al = \pi^*\om.$ Consider $(Y,\xi)$ for $\xi =\ker(\al)$ as a contact manifold.

\begin{thm}\label{thm:preq}
Arnol'd's chord conjecture holds for every prequantization $Y$ of a closed symplectically aspherical symplectic manifold. It also holds for prequantizations of closed monotone symplectic manifolds of minimal Chern number at least $2$ and for Liouville-fillable prequantizations. For an arbitrary prequantization $Y$ of a closed symplectic manifold Arnol'd's chord conjecture holds for all relatively spin Legendrians $K.$\end{thm}

\begin{rmk} A few remarks regarding Theorem \ref{thm:preq}.
\begin{enumerate}[label = (\roman*)]
\item  Liouville-fillable prequantizations include the cosphere bundles $S^*Q$ of compact rank-one symmetric spaces: $Q \in  \{S^n, \C P^n, \mathbb H P^n, \mathbb O P^2\}.$
\item A submanifold $K$ of $Y$ being relatively spin means that $K$ is orientable and $w_2(K) = b|_K$ for a class $b \in H^2(Y;\F_2).$ For instance, all $K$ homeomorphic to a sphere are relatively spin.
\item The argument in the case of a general prequantization relies on virtual perturbation techniques of Daemi-Fukaya and their notion of RGW compactifications \cite{DF1,DF2,DF3}. The other cases only use classical transversality techniques.
\item We expect that the technical assumption on $K$ being relatively spin in the case of general prequantizations can be removed by developing the FOP perturbations approach \cite{BX1,BX2, BX3, R1, Rab} or the $\F_2$ Morava $K$-theory approach \cite{AB-Morava} in the setting of RGW compactifications \cite{DF1,DF2,DF3}. This will be the subject of future work.
%c) The symplectically aspherical case is proven in a few diffirent ways. %We expect that it should also follow from the recent approach of Oh \cite{Oh-WAC} to the chord conjecture. %However, the general case likely does not follow from this approach.
\item Theorem \ref{thm:preq} a special case of a more general result, Theorem \ref{thm:gen} below. Note that due to the Boothby-Wang theorem, prequantizations are exactly the contact manifolds admitting a contact form $\al_0$ whose Reeb flow is periodic of period $1$ and provides a free $S^1$-action. We expect Theorem \ref{thm:preq} to extend to all contact manifolds with periodic Reeb flow. These are prequantizations of symplectic orbifolds, and as they are weakly non-orderable (Definition \ref{def:ord}), it remains to prove that they are of displacement-type (see Definition \ref{def:chek} and Theorem \ref{thm:gen} below). We expect to do this in future work, for now restricting ourselves to the following result which directly follows from Theorem \ref{thm:gen} and Proposition \ref{prop:chek} below.\end{enumerate}
\end{rmk}

\begin{thm}\label{thm: Brieskorn}
Arnol'd's chord conjecture holds for all links of isolated weighted homogeneous hypersurface singularies $Y = p^{-1}(0)  \cap S^{2n+1} \subset \C^{n+1}.$ In particular, it holds for Brieskorn manifolds \[Y = \{  z_0^{a_0} + \ldots + z_n^{a_n} = 0 \} \cap S^{2n+1} \subset \C^{n+1},\] for integers $a_j \geq 1,$ $0 \leq j \leq n.$ 
\end{thm}

Indeed, these manifolds are weakly non-orderable, as their standard contact form $\al = \la_{\C^{n+1}|_Y}$ has periodic Reeb flow (see \cite{Brieskorn1}). At the same time they are Liouville-fillable and therefore of displacement-type. Of course, only $a_j \geq 2$ for all $0 \leq j \leq n$ provide new examples, as the other cases are all contact boundaries of subcritical Weinstein domains and hence are covered by \cite{Mohnke}. All the standard fillings of Brieskorn manifolds with $a_j \geq 2$ are known to have non-trivial symplectic cohomology (see \cite{Brieskorn1, Brieskorn2}) and hence Theorem \ref{thm: Brieskorn} does not follow from \cite{Zhou-minimal}.  

Finally, we sharpen Theorem \ref{thm:preq} quantitatively as follows. 

\begin{cor}\label{cor:preq}
In Theorem \ref{thm:preq}, if $\al = f \al_0,$ for $\al_0$ the standard contact form with periodic Reeb flow of period $1,$ the chord $c:[0,T] \to Y$ can be chosen with \[T \leq 2\max_Y f.\] 
\end{cor} 

\begin{rmk}
We note that the multiplicative factor of $2$ on the right-hand side is not optimal and we expect it to be possible to remove it by more careful cutoff considerations.\end{rmk}

\medskip

To describe the outline of the proof, we start with a few definitions.

\begin{df}\label{def:rational}
Let $(W,\om),$ $\om = d\la,$ be an exact symplectic manifold. We call a Lagrangian submanifold $L \subset W$ {\em totally rational} if $[\la|_L] \in H^1(L;\R)$ is a rational cohomology class, that is it lies in the subgroup $\rho \cdot H^1(L;\Z)/\mrm{tor}$ for a real number $\rho \geq 0,$ where $\mrm{tor}$ is the torsion subgroup. We call the supremum of such $\rho$ the rationality constant $\rho_L \in (0,\infty) \cup \{+\infty\}$ of $L.$ \end{df}

Note that $\rho_L > 0$ is finite if and only if $[\la_L] \neq 0,$ in which case it is determined by the condition that the period group of $\la_L$ satisfies \[\langle [\la_L], H_1(L;\Z)\rangle = \rho_L \cdot \Z.\] Moreover $\rho_L = +\infty$ if and only if $[\la_L] = 0,$ that is, when $L$ is an exact Lagrangian.
%We call a Lagrangian $L$ in the symplectization $(SY, d\la)$ of a contact manifold with its canonical primitive one-form $\la$ 

Recall that the Hofer displacement energy of a compact subset $B \subset W$ of a symplectic manifold $W$ (which will be open in our case) is defined as \[e(B) = \inf \ell(F),\] \[\ell(F) = \int_0^1 (\max_W F_t -\min_W F_t) \,dt,\]  for $F \in C^{\infty}_c({[0,1] \times W})$ a Hamiltonian with compact $W$-support $\mrm{supp}(F) = \cup_t \mrm{supp}(F_t),$ $F_t(x) = F(t,x),$ such that the Hamitonian flow $\phi^t_F$ of $F,$  generated by the time-dependent vector field $X^t_F,$  $\iota_{X^t_F} \om = - dF_t,$ displaces $B:$ \[\phi^1_F(B) \cap B = \emptyset.\]

\begin{df}\label{def:chek}
We say that a closed contact manifold $Y$ is of displacement-type if the displacement energy of a totally rational Lagrangian $L$ in the symplectization $SY$ of $Y$ is at least $\rho_L.$ We call it of $\Q$ displacement-type if the same is true for relatively spin such $L.$\end{df}

A few examples of such manifolds are provided by the following proposition.

\begin{prop}\label{prop:chek}
Every prequantization of a closed symplectic manifold is of $\Q$ displacement type. The following are also of displacement type: every prequantization of an aspherical symplectic manifold or of a mononotone symplectic manifold of minimal Chern number at least two; every Liouville-fillable contact manifold; every hypertight contact manifold; every contact manifold $Y$ with a symplectically aspherical filling $W$ such that $\pi_1(Y) \to \pi_1(W)$ is injective. \end{prop}

\begin{rmk}
Note that we require a sharp lower bound $e(L) \geq \rho_L$ as in the definition of displacement-type for the proof of Theorem \ref{thm:gen} below to go through. Such lower bounds were produced by Chekanov \cite{Chek} in geometrically bounded symplectic manifolds, generalizing the work of Polterovich \cite{Polterovich-displacement} proving the bound $e(L) \geq \frac{\rho_L}{2}.$ In contrast, for Mohnke's original argument \cite{Mohnke} as well as for \cite{Zhou-minimal}, a linear bound in $\rho_L$ is sufficient (or in fact any lower bound in $\rho_L$ which tends to infinity as $\rho_L \to \infty$).
\end{rmk}

Certainly not all contact manifolds are of displacement type or of $\Q$ displacement type. For instance overtwisted $Y$ of dimension $\geq 5$ are not of displacement type, as by Murphy \cite{murphy-exact} $SY$ has compact exact Lagrangians $K.$ Moreover, these Lagrangians have vanishing displacement energy by Sikorav and Chekanov \cite{Chek}. Further examples are provided in Murphy  \cite{murphy-exact} in dimension $\geq 5$ and by Dimitroglou-Rizell \cite{dr-exact} in dimension $3.$

%(like in Sikorav/Chekanov or Dylan/Lukas). Georgios also had some 3d examples.

%Definition 1: call a compact contact manifold Y displacement-type if the displacement energy of a rational Lagrangian L in the symplectization SY is at least the rationality constant of L. (In fact can ask for less, e.g. a. Chekanov's threshold instead of the rationality constant = "minimal area of holomorphic disk etc", or even less than thisb. actually the Lagrangians we need this for are quite special)
%Remark 1: 
%a. examples of displacement-type manifolds include:a1. hypertight Y, for example prequantizations of integral symplectically aspherical manifolds, and a2. Y with aspherical incompressible fillings (by incompressible I mean that the inclusion of Y into the filling is injective on \pi_1), e.g. S^*S^n for n\geq 3.
%a3. I also have an idea, which I am happy to outline, towards proving that all prequantizations are also displacement-type.
%b. non-examples include overtwisted Y of dim \geq 5 as by Murphy SY has compact exact Lagrangians and these can be very cheaply displaced (like in Sikorav/Chekanov or Dylan/Lukas). Georgios also had some 3d examples.

\begin{question}
How to characterize the class of displacement-type manifolds?
\end{question}

\begin{conjecture}\label{conj}
All tight contact manifolds should be of displacement type.
\end{conjecture}

\begin{rmk}
It appears that Conjecture \ref{conj} should hold whenever the contact homology algebra of $Y$ admits an augmentation. This statement and further results related to Theorems \ref{thm:preq} and \ref{thm: Brieskorn} are expected to appear in a forthcoming work of Zhengyi Zhou based on different methods that the ones presented in this paper.
\end{rmk}

% whose contact homology algebra admits an augmentation 

%Question: how to characterize the class of displacement-type manifolds?Conjecture: all Y whose contact homology algebra admits an augmentation should be in this class.

We make the following definition to fix terminology in a way convenient for this paper. Note that this terminology is not uniform in the literature: we choose it to agree with the original terminology from \cite{EKP,EP}. Call a loop of contactomorphisms $(\theta^t)_{t \in S^1}$ positive if its contact Hamiltonian satisfies $g(t,x) > 0$ for all $t \in [0,1]$ and $x \in Y.$

\begin{df}\label{def:ord}
We call a compact contact manifold $Y$ {\em weakly non-orderable} if there is a positive loop of contactomorphisms (and {\em non-orderable} if there is a positive contractible such loop). Likewise, we call it {\em strongly orderable} if there is no positive loop of contactomorphisms (and {\em orderable} if there is no contractible such loop).
\end{df}

An important result of \cite{EP} is that in the above definition one can equivalently assume the existence of a non-negative loop $(\theta^t)_{t \in S^1},$ that is its contact Hamiltonian satisfies $g(t,x) \geq 0$ for all $t \in [0,1]$ and $x \in Y,$ and therefore strong orderability and  orderability are equivalent to the existence of natural partial orders on the identity component $\mathrm{Cont}_0(Y,\xi)$ of the group of contactomorphisms of $(Y,\xi)$ and its universal cover respectively. 

%\red{CONTINUE FROM HERE}
%Definition 2 (just to fix terminology): call a compact contact manifold Y non-orderable if there is a positive loop of contactomorphisms (strongly non-orderable will be a positive contractible loop); call it strongly orderable if there is no positive loop of contactomorphisms (and orderable if there is no contractible such loop).
%
We are ready to state the most general result in this paper.

\begin{thm}\label{thm:gen}
The chord conjecture holds for every weakly non-orderable displacement-type contact manifold. It holds for relatively spin Legendrians for weakly non-orderable $\Q$ displacement-type contact manifolds.
\end{thm}

\begin{rmk} A few further remarks.
\begin{enumerate}[label = (\roman*)]
\item In view of Proposition \ref{prop:chek}, Theorem \ref{thm:gen} implies Theorem \ref{thm:preq}. 
\item It also recovers Mohnke's result \cite{Mohnke} for $\dim Y \geq 5$ by Proposition \ref{prop:chek} and the main result of \cite{HS-no}. 
\item Finally, the proof of Corollary \ref{cor:preq} shows that the length of the minimal chord in Theorem \ref{thm:gen} is bounded from above by $2S_0$ where $S_0$ is the infimum of all $S>0$ with $\phi^S_R$ generated by a {\em non-negative} contact Hamiltonian $h$ of contact Hofer length $\ell(h) < S.$ This suggests that the contact Hofer geometry of non-negative paths is a potentially interesting subject to be investigated in the future (see also \cite{EKP, Dahinden, Dahinden-old, H-ord}). 
\item Note that \cite{BC-chord, GZ-chord} prove the chord conjecture for certain manifolds, notably $Y = S^*T^n,$ which are strongly orderable by \cite{EKP, ChernovNemirovski}. Our result is thus complementary to theirs. Indeed, the only overlap between our result and that of \cite{BC-chord} appears to be the special case $Y = S^*Q$ for $Q$ a compact rank one symmetric space. %of both results.  
\end{enumerate}
\end{rmk}

\begin{rmk}
It is an interesting question whether the proof of the chord conjecture in dimension $3$ \cite{HT1, HT2} can be upgraded to be uniform in the above sense: there exists a constant $C(\al)$ depending only on the contact form $\al$ on $Y^3$ such that the length of the minimal chord of every Legendrian $K \subset Y$ is at most $C(\al).$ Note that the above proof as well as those in \cite{Mohnke, Zhou-minimal, GZ-chord, BC-chord} do produce such uniform bounds.
\end{rmk}

Recall that given a contact form $\alpha$ on a closed contact manifold $Y,$ the contact Hofer norm \cite{S-ch} of a contactomorphism $\phi \in \mathrm{Cont}_0(Y)$ is given as \[||\phi||_{\al} = \inf \ell(h),\] \[\ell(h) = \int_0^1 \max_Y |h_t|\,dt,\] $h_t(x) = h(t,x),$ where the infimum runs over all contact Hamiltonians $h$ with respect to $\al$ whose contact flow $\phi^t_h$ generated by the vector field $X^t_h$ defined by $\al(X^t_h) = h_t, \iota_{X^t_h} d\al|_\xi = -dh_t|_\xi,$ satisfies $\phi^1_h = \phi.$

A brief summary of the proof strategy of Theorem \ref{thm:gen} is now as follows. Using methods of \cite{S-ch} we expect to prove that the Reeb flow of a contact form $\alpha$ on a displacement-type manifold admitting a Legendrian without Reeb chords must have contact Hofer norm $||\phi^T_R||_{\al} = T$ for all $T \geq 0.$ By \cite{H-ord, HS-no}, this implies that $Y$ is strongly orderable, a contradiction. However, for technical reasons, to relate the Hamiltonian displacement energy in $SY$ to the contact Hofer length, it turns out useful to work with non-negative contact paths constructed explicitly. With this technical point in mind, the proof proceeds similarly to the above strategy, losing a factor of $2$ in the quantitative estimate.

%Theorem: The chord conjecture holds for every displacement-type non-orderable contact manifold.
%Corollary: The chord conjecture holds for prequantizations of (integral) symplectically aspherical manifolds, for S*S^n, n\geq 3 (actually for S^*M for M Zoll, simply connected, and of dimension \geq 3, like CP^m, m\geq 2 etc). It should also hold for all prequantizations.

%%We claim that this implies that, calculated with respect to $\alpha,$ the contact Hofer norm of $\phi^T_R$ is $|| \phi^T_R || = T.$ Indeed, suppose that $\phi^{T} = \psi^1_h$ for a contact Hamiltonian $h=h(t,x)$ 

\section{Proofs}

\subsection{Proof of Theorem \ref{thm:gen}}
Let $Y$ be a displacement-type weakly non-orderable contact manifold and $K$ a Legendrian without Reeb chords for a contact form $\alpha.$ For time $T$ look at $SY$ and Mohnke's Lagrangian $L$ built out of the images $K_t$ of $K$ under the Reeb flow $\phi^t = \phi^t_R$ of the contact form for $t \in [0,T-\delta],$ for $\delta>0$ very small, Liouville scalings of the endpoints by $e^{-\tau B}$ in $SY,$ for $B$ very large, and $\tau \in [0,1],$ and the $e^{-B}$ scalings of the $K_t.$ For $B$ very large, $\delta$ very small, and a carefully performed smoothing at the corners, this $L$ is totally rational with rationality constant $\rho_L$ arbitrarily close to $T,$ say $T-\eps<\rho_L<T$ for arbitrarily small $\eps>0.$ %By eventually passing to a limit, let's pretend that it is actually T. 

As $Y$ is displacement-type, the displacement energy of $L$ is at least $\rho_L.$ However $L$ is displaced by the canonical lift of $\phi^{T}_R$ to a Hamiltonian diffeomorphism of $SY.$ Let $(\theta_t)_{t \in [0,1]}$ be a positive loop of contactomorphisms based at the identity and let $g(t,x) > 0,$ $t \in [0,1], x\in Y$ be its positive contact Hamiltonian. If $T \geq \max_{[0,1] \times Y} g$ then \[h(t,x) = T - g(1-t, \phi^{-tT}_R x) \geq 0\] is a non-negative contact Hamiltonian generating $\phi^T$ with contact Hofer length \begin{equation}\label{eq: ell h} \ell(h) = \int_0^1 \max_Y |h_t| dt < T,\end{equation} for $h_t(x) = h(t,x).$ Indeed, as $\theta_1^{-1} = id,$ $h$ generates the contact flow $(\phi^{tT} \theta_{1-t})_{t\in [0,1]}$ from the identity to $\phi^T.$ Consider the Hamiltonian $H_t(\rho,x) = H(t,\rho, x) = e^\rho h(t,x)$ of the canonical lift of the contact flow $(\phi^t_h)$ of $h$ to $SY.$ Note that $H_t \geq 0$ for all $t \in [0,1].$

Like in \cite[Proof of Proposition 11, Proposition 41, Lemma 42]{S-ch}, use Usher's trick to change $H_t$ to $F_t \geq 0$ with the same time-one map, and the property that on $SY$ the Hofer length of $F_1(t,x) = F(t,x) \chi((\phi^t_F)^{-1} x),$ where $\chi$ is a cutoff which is $1$ on $A \times Y$ in $SY$ for $A = [e^{-S},1]$ and $0$ outside the $\eps$-neighborhood interval, is bounded by $e^{\eps} \ell(h).$ Note that as $(F_1)_t \geq 0,$ and it is compactly supported, \[\max_{SY} (F_1)_t - \min_{SY} (F_1)_t = \max_{SY} |(F_1)_t |\] for all $t \in [0,1].$ Therefore \[\ell(F_1) = \int_0^1 \max_{SY} |(F_1)_t|\, dt.\] However, $\phi^1_{F_1}(L) = \phi^1_F(L) = \phi^1_H(L),$ which is disjoint from $L.$ Hence, as $Y$ is displacement-type we get that \[\ell(h) \geq e^{-\eps} \rho(L)>e^{-\eps}(T-\eps)\] for all $\eps>0.$ Taking the limit as $\eps \to 0$ yields $\ell(h) \geq T$ in contradiction to \eqref{eq: ell h}. This finishes the proof in the displacement-type case.

%and therefore \[||\phi^T||+\eps \geq ||\phi^{T+\eps}|| \geq e^{-\eps}(T-\eps).\] Hence \[||\phi^T|| \geq e^{-\eps}(T-\eps) -\eps,\]
%as $\phi^1_H(\rho, x) = (\rho, \phi^T(x))$ as a map of $SY.$ 
%Now by \cite{H-ord, HS-no} this implies that $Y$ is strongly orderable, which is a contradiction to the assumption that Y is non-orderable. 

The proof applies verbatim for $\Q$ displacement-type manifolds, replacing ``Legendrian" by ``relatively spin Legendrian" everywhere.

\qed

\subsection{Proof of Corollary \ref{cor:preq}}

It is only required to estimate from above the infimum $S_0$ of the set of $S>0$ with $\phi^S$ being generated by a non-negative Hamiltonian $h$ with $\ell(h)<S.$ The above proof goes through and provides a chord of length $T \leq 2S_0.$ Indeed, for the proof to work for a given $T$ it is sufficient that $K$ has no chords of length at most $2T$ to construct $L$ and its displacement. %$|||| < S.$ %Every point of $L$ is of the form $(\rho, \phi_t(x))$ for $x \in K$ and $0 \leq t < T.$ On the other hand the Hamiltonian lift of $\phi_T$ to $SY$ maps $(\rho, x)$ to $(\rho, \phi_T(x)).$ Hence $\phi_T \phi_{t_1} x = \phi_{t_2} x$

The contact Hamiltonian with respect to $\al$ of the Reeb flow $(\psi^t)_{t \in \R/\Z}$ of $\al_0$ is given by $f>0.$ Therefore for $S = \max_Y f,$ $\phi^S$ is generated by the non-negative Hamiltonian $h(t,x) = S - f(\phi^{(1-t)T}_R x).$ Moreover \[\ell(h) \leq S-\min_Y f < S.\] Therefore $S_0 \leq \max_Y f$ and we obtain a chord of length $T \leq 2\max_Y f.$

% the contact Hofer norm (with respect to $\al$) satisfies \[|| \psi^{-s}\phi^S || = || \phi^{-S} \psi^s || \leq \int_0^1 \max_Y |-S + sf \circ (\phi^t)^{-1}|\,dt \] for $s=1$ this yields \[ || \phi^S || = || \psi^{-1}\phi^S || \leq \max_Y |-S + f| = \max \{ |S-\max_Y f|, |S-\min_Y f| \}.\] %Now $\phi^S$ is generated by a po %Set $a=\min_Y f \leq b= \max_Y f.$ Then \[d(S) := \max \{ |S-\max_Y f|, |S-\min_Y f| \} = |S-\frac{a+b}{2}| + \frac{b-a}{2}\] which, solving the inequality $d(S)<S,$ provides \[S_0 \leq \frac{b}{2} = \frac{\max_Y f}{2}.\]  %Therefore 

\subsection{Proof of Proposition \ref{prop:chek}}

Let us first prove the case of Liouville-fillable contact manifolds $Y.$ It is essentially contained in Mohnke \cite{Mohnke} and Chekanov \cite{Chek}. Let $F$ be a compactly supported Hamiltonian on $SY$ with $\phi^1_F$ displacing $L.$ Let $(W,\la_W)$ be a Liouville filling. A contact form $\al$ on $Y$ and the Liouville flow on $W$ provide an embedding $i: SY \to W$ with $i^*\la_W = \la_{SY},$ where $\la_{SY} = e^\rho \al$ is the standard Liouville form on $SY.$ By means of $i$ we can consider $L$ as a Lagrangian submanifold of $W,$ and using extension by zero, $F$ a compactly supported Hamiltonian on $W.$ Let $\om_W = d\la_W$ be the symplectic form on $W,$ and $J$ be an $\om_W$-compatible almost complex structure on $W$ which is of SFT type outside a compact subset of $W$ (which can be taken to lie in its infinite end): it sends the Liouville vector field $Z_W$ on $W$ to the Reeb vector field $i_*(R_\al),$ preserves $\ker \la_W$ and is invariant under the positive Liouville flow. This makes $(W,\om_W, J)$ into a geometrically bounded symplectic manifold.

By Chekanov's theorem \cite{Chek}, we obtain that \[\ell(F) \geq \min \{ \om_{W}(u)\}\] where $u$ run over all (stable) non-constant $J$-holomorphic disks with boundary on  $L$ and (stable) non-constant $J$-holomorphic spheres in $E$ passing via $L$ and $\om_W(u) = \int_u \om_W.$ (This is considered with the caveat that the minimum of an empty set is taken to be $+\infty.$) However, by Stokes' theorem and exactness of $\om_W$ no such holomorphic spheres exist and for every such disk \[u: (\D, \partial \D) \to (W, L),\] \[\om_W(u) = \int_{\partial u} \la_W > 0,\] where $\partial u = u|_{\partial \D}.$ Moreover, by total rationality of $L,$ and $\partial u$ taking values in the image of $i,$ \[\int_{\partial u} \la_W = \int_{i^{-1} \circ \partial u} \la_{SY} \in \rho_L \cdot \Z.\] In particular $\om_W(u) \geq \rho_L.$ Hence $\ell(F) \geq \rho_L,$ which finishes this proof.

The case of a symplectically aspherical filling is similar. First, by scaling $L$ by the Liouville flow on $SY$ we may assume that $L$ is contained in the positive end $S =(a,\infty) \times Y \subset SY$ which embeds into $W$ as before. Call the embedding $i.$ We proceed to rule out holomorphic spheres by the symplectically aspherical condition $[\om_W]|_{\pi_2(W)} = 0.$ To estimate $\om_W(u)$ for a non-constant $J$-holomorphic disk $u: (\D, \partial \D) \to (W, L),$ we proceed as follows. First $\om_W(u) > 0.$ Second the loop $\del u$ is contained in the image of $i,$ while it is contractible in $W.$ By the injectivity condition on $\pi_1,$ $i^{-1} \circ \del u$ is contractible in $S.$ Let $v: (\D, \partial \D) \to (SY, L)$ be a contracting disk. By the symplectically aspherical condition on $W$ again, \[\om(u) = \om(v) = \int_{\partial v} \la_{SY} \in \rho_L \cdot \Z.\] Therefore $\ell(F) \geq \rho_L,$ as before.

In the hypertight case, we might proceed following Chekanov's argument \cite{Chek}, Oh's approach to it \cite{Oh-disjunction}, or Floer homology in action windows as in \cite{CS-JDG, KS-GT} for example (or truncated Floer homology over the Novikov ring). It is convenient to work with the third approach for details of which we refer the reader to \cite{CS-JDG, KS-GT}. We only sketch the salient extra points in the current setup related to working directly in the symplectization.
Suppose that $\ell(F) < \rho_L$ and $\phi^1_F$ displaces $L$ in $SY.$ Let $\eps>0$ be sufficiently small so that $\ell(F)<\rho_L-2\eps.$ Consider the window $I_c = (-\rho_L, \eps) + c \subset \R$ for a constant $c \in \R.$ Let \begin{equation}\label{eq:window} CF(L;F, \cl D)^{I_c}\end{equation} be the Floer complex over $\F_2$ generated by Hamiltonian chords of $F$ from $L$ to $L$ contractible to $L$ with a very small perturbation datum to achieve transversality for curves contained in $SY$, which is supported in $\{\rho \geq \rho_F\} \subset SY$ where $\rho_F = \min_{\mathrm{supp} F} \rho.$ Note that only energy $\leq \rho_L - \eps$ Floer strips enter into the calculation of the differential \[d: CF(L;F, \cl D)^{I_c} \to CF(L;F, \cl D)^{I_c}.\] Of course, one needs to show that $d^2 = 0.$ Choosing an SFT-type almost complex structure on $SY$ with respect to a hypertight contact form at the negative end and with respect to an arbitrary contact form at the positive end of $SY,$ a standard SFT compactness argument (see e.g. \cite{AlbersFuchsMerry, MeiwesNaef}) shows that all energy at most $E = \rho_L -\eps$ Floer strips are contained in $\{\rho \geq \rho'_F\}$ for suitable finite $\rho'_F \leq \rho_F$ and the same is true of all holomorphic disks on $L,$ as otherwise a contractible Reeb orbit would be produced in the SFT limiting configuration. The curves are also all contained in $\{\rho \leq R\}$ for suitable large $R$ by the maximum principle in the positive end. Then by total rationality, the areas of the non-constant holomorphic disks is at least $\rho_L$ whence $d^2 = 0$ on $CF(L;F, \cl D)^{I_c}.$ Furthermore, for small perturbation data $\cl D', \cl D, \cl D''$, the usual continuation maps \[CF(L;0,\cl D')^{I_0} \to CF(L;F,\cl D)^{I_{c_+}} \to CF(L;0, \cl D'')^{I_{\rho_L-\eps}},\] \[c_+ = \int_0^1 \max F_t \,dt\] counting only curves in $SY$ are well-defined chain maps and are non-trivial in homology (again all curves lie in  $\{\rho \geq \rho''_F\}$ for suitable finite $\rho''_F \leq \rho'_F$). Indeed, they induce the interval comparison map $HF(L;0,\cl D')^{I_0} \to HF(L;0, \cl D'')^{I_{\rho_L-\eps}}$ which in turn commutes with the injection of the homology $H_*(L;\F_2)$ at the zero level. Note that this is a contradiction as $CF(L;F,\cl D)^{I_{c_+}}  = 0$ since $\phi^1_F(L) \cap L = \emptyset,$ which implies $HF(L;F,\cl D)^{I_{c_+}}  = 0.$ Therefore $\ell(F) \geq \rho_L$ as desired.

Now prequantizations of integral symplectically aspherical manifolds $(M,\om)$ are well-known to be hypertight, for the prequantization contact form \cite{AlbersFuchsMerry}. Hence such prequantizations are of displacement type. For prequantizations of monotone symplectic manifolds of minimal Chern number at least two, there is a similar SFT argument for working only in the symplectization which additionally involves considerations of index: see \cite{SUV} (and \cite{ASZ}) for the relevant index estimates. 

Let us now present a somewhat different argument in the symplectically aspherical case, as it appears to generalize most conveniently (it also applies in the monotone case with minimal Chern at least two following \cite[Remark 3.11]{SUV}). Let $\pi: Y \to M$ be a prequantization space of a closed (integral) symplectic manifold $(M, \om).$ Let $\al$ denote the standard contact form on $Y$ such that $d\al = \pi^*\om.$ Then $Y$ is the hypersurface of unit-norm vectors in a complex Hermitian line bundle $\pi: E\to M.$ In other words if we set $r^2: E \to \R$ to be the norm-squared function of the Hermitian metric, then $Y = \{r^2 = 1\}.$ For every $\eps >0,$ \[\om_{\eps} = \eps\pi^* \om + d(r^2 \al)\] is a symplectic form on $E.$ Let $J_M$ be an $\om$-compatible almost complex structure on $M$ and let $j_E$ denote the complex structure on $E.$ There exists an almost complex structure $J$ on $E$ such that $\pi: E \to M$ is $(J,J_M)$-holomorphic and $J = j_E$ on the vertical subbundle $V = \ker(D\pi) \cong \pi^*E$ of $TE.$ Indeed, we should only remark that the horizontal subbundle $H \subset TE,$ $H = \ker(\al)$ on $T(E\setminus 0_E)$ and $H = T(0_E)$ is isomorphically indentified with $\pi^*TM$ by $D\pi,$ where $M \cong 0_E \subset E$ is the zero section. Note that $J$ is compatible with $\om_{\eps}$ for all $\eps > 0.$ Note that $(E,\om_{\eps}, J)$ makes $E$ into a tame symplectic manifold. Moreover, the subset \[S_{\eps}Y = (\log(\eps),\infty) \times Y \subset SY\] with its canonical symplectic form $\om_{SY} = d\lambda,$ $\lambda = e^{\rho} \al,$ $\rho: SY = \R \times Y \to \R$ the projection to the first coordinate, embeds symplectically into $(E,\om_{\eps})$ by \[i_{\eps}: (\log(\eps),\infty) \times Y \to E,\] \[i_{\eps}(\rho, y) = ry\] for \[r = (e^{\rho} - \eps)^{1/2}.\] The image of $i_{\eps}$ is the symplectic submanifold \[E_{\eps} = (E \setminus \{0_E\}, \om_\eps)\] of $(E,\om_{\eps}).$ In fact, $(E,\om_{\eps})$ can be identified with the symplectic cut of $SY$ at the level $\log(\eps)$ with respect to the canonical Hamiltonian circle action generated by the Hamiltonian $H(\rho,x) = \rho$ on $SY.$

Now let $L \subset SY$ be a totally rational Lagrangian and $F$ a compactly supported Hamiltonian on $SY$ with $\phi^1_F$ displacing $L.$ Choose $\eps>0$ sufficiently small so that $\mrm{supp}_M(F) \subset S_{2\eps} Y \subset S_{\eps} Y.$ Consider $F$ as a Hamiltonian on $(E,\om_{\eps})$ via the identification $i_{\eps}:S_\eps Y \to E_{\eps}$ and extension by zero over $0_E$ to $(E,\om_\eps).$ We rely on the following key observation.

\begin{lma}\label{lma: int zero}
Let $A \in H^D_2(E, i_\eps(L); \Z)$ be a disk relative homology class. Suppose that the intersection number of $A$ and the zero-section divisor $0_E$ of $E$ satisfies $A \circ 0_E = 0.$ Then \[\om_\eps(A) = \langle [\lambda|_L], (i_\eps)^{-1}_*(\del A) \rangle \in \rho_L \cdot \Z\] for the connecting homomorphism $\del: H^D_2(E, i_\eps(L); \Z) \to H_1(i_\eps(L); \Z).$
\end{lma}

\begin{proof}
Suppose that $A$ is represented by a disk $u:(\D, \del \D) \to (E, i_\eps(L)).$ It is well-known (see \cite[Sections 3.4, 5.1]{Frauenfelder-hab} and also \cite[Lemma 20]{S-VZ} with the difference that in our conventions $d\al = \pi^*\om$ instead of $d\al = -\pi^*\om$) that \[A \circ 0_E = \int_{\del u} \al - \int_u \pi^*\om.\] Hence the condition of the lemma implies that \[\int_{\del u} \al = \int_u \pi^*\om.\] We now calculate \[ \om_\eps(A) = \eps \int_u \pi^*\om + \int_{\del u} r^2 \al =  \int_{\del u} (\eps+r^2) \al = \int_{(i_\eps)^{-1} \circ \del u} \lambda.\] The result now follows by evident identifications and the total rationality of $L.$
\end{proof}

Suppose that $(M,\om)$ is symplectically aspherical, and $||F|| < \rho_L.$ Considering Floer complexes in short action windows and continuation maps as above between them, together with Lemma \ref{lma: int zero}, shows that there is a well-defined complex where only curves inside $E_{\eps}$ are considered. Indeed, consider only curves and Floer strips with zero intersection number with the zero section. For small perturbations of the above almost complex structures away from the zero section, each such curve consists of a root component and bubble trees attached at the intersection points which must be contained in the zero section in view of the maximum principle. By the positivity of intersection of the root component with $0_E$ and the total zero-intersection condition with $0_E,$ either the root component is disjoint from the zero section and there are no bubbles or at least one bubble in all the bubble trees must be non-constant. However, the zero section is symplectically aspherical and hence the bubble configuration must be constant. Therefore the only possibility is that there is only a root component which lies inside $E_\eps.$ The property $d^2 = 0$ now follows from Lemma \ref{lma: int zero} or in fact by direct calculation in $E_{\eps}$. The proof now proceeds identically to the one in the hypertight case. 

%The only possible sphere bubbling of a sequence of such objects outside $0_E$ would be into the zero section, and the bubble configuration would be non-constant by positivity of intersection of the root component with $0_E$ and the total zero-intersection condition. However, the zero section is aspherical and hence the bubble configuration must be constant. The property $d^2 = 0$ follows from Lemma \ref{lma: int zero} or in fact by direct calculation, once we know that no curve can intersect 

Finally, in the general case, for relatively spin totally rational Lagrangians, the methods of \cite{DF1, DF2, DF3}, see also \cite[Section 4]{DF-survey}, together with Lemma \ref{lma: int zero} immediately provide complexes in short action windows as well as continuation maps as above, based on counting so-called RGW configurations of curves, whose homology classes have zero intersection numbers with $0_E.$ Then the above proof structure proceeds analogously.

We provide further details regarding the general case. Note that in the general one cannot rule out configurations involving bubble trees lying in the zero section. Daemi and Fukaya instead count certain such configurations, which in particular have zero total intersection number with the zero-section, using virtual perturbation techniques. They introduce special compactifications of such objects, the RGW compactifications, which are different from the usual stable curve compactifications, prove that they admit Kuranishi structures and describe their normalized boundaries \cite[Theorems 2.8, 2.9, Corollary 2.10]{DF3}. This allows them to define the differential $d$ by virtually counting the points in suitable dimension $0$ RGW moduli spaces \cite[Definition 2.13]{DF3}. We make the same definition but focus on an action window shorter than $\rho_L,$ similarly to \eqref{eq:window} above.

The fact that, in our situation $d^2 = 0$ relies on the results of \cite{DF3} regarding these compactifications together with Lemma \ref{lma: int zero}. We follow the same proof as for \cite[Theorem 2.16, Section 2.4]{DF3} until showing that the contributions of the components of Items (2) and (3) of \cite[Theorem 2.8]{DF3} vanishes: see discussion near \cite[Equation 2.53]{DF3}. These Items involve bubbling off of ``disk-like" RGW configurations. To do this, instead of relying on monotonicity outside of the divisor as in \cite{DF3}, we observe that all the RGW configurations contributing to these components non-trivially must have strictly positive energy. Hence, as the energy of ``disk-like" configurations is equal to their symplectic area, by Lemma \ref{lma: int zero} they must have energy at least $\rho_L.$ However, as the length of our action window is strictly smaller than $\rho_L,$ these contributions do not affect any counts in this action window. Therefore $d^2 = 0$ in our case. The fact that continuation maps are chain maps and compositions induce interval comparison maps again follow from the properties of RGW compactifications as in \cite{DF1, DF2, DF3}, together with ruling out contributions of ``disk-like" elements in short windows by Lemma \ref{lma: int zero}.

\section*{Acknowledgments} I thank Mohammed Abouzaid, Peter Albers, Shaoyun Bai, Filip Bro\'ci\'c, Dylan Cant, Octav Cornea, Jakob Hedicke, Asaf Kislev, Leonid Polterovich, and Frol Zapolsky for interesting discussions and collaborations on related subjects and Zhengyi Zhou for an interesting discussion about his ongoing project. This work was supported by an NSERC Discovery grant, an FRQNT Teams grant, and by the Courtois chair in fundamental research.

\bibliographystyle{abbrv}
\bibliography{refs}

\end{document}